\documentclass[11pt, a4paper]{amsart} % A4 format
\usepackage[english]{babel}
\usepackage{amsmath,amsfonts}
\usepackage{amsthm}
\usepackage{amssymb}
\usepackage[alphabetic]{amsrefs}
\usepackage{url}
\usepackage[all]{xy}
\usepackage{graphicx}
\usepackage{psfrag}
\usepackage{array}
\usepackage{enumerate}
\usepackage{mathtools}
\usepackage{color} %
\usepackage{dsfont} 
\usepackage{mathrsfs}
\usepackage{lipsum}
\usepackage{relsize} %for \mathlarger{} 
\usepackage{upgreek} %for \uptau, etc.
\usepackage{caption}
\usepackage{pgf,tikz}
\usetikzlibrary{arrows}
\usepackage{xargs} %For xcommand and other stuff

\usepackage{xcolor}
\definecolor{rouge}{rgb}{0.85,0.1,.4}
\definecolor{bleu}{rgb}{0.1,0.2,0.9}
\definecolor{violet}{rgb}{0.7,0,0.8}
\definecolor{noir}{rgb}{0,0,0}
\usepackage[colorlinks=true,linkcolor=bleu,urlcolor=noir,citecolor=rouge]{hyperref} 

\theoremstyle{theorem}
\newtheorem{theorem}{Theorem}[section]
\newtheorem{question}[theorem]{Question}
\newtheorem{proposition}[theorem]{Proposition}
\newtheorem{corollary}[theorem]{Corollary}

\theoremstyle{definition}
\newtheorem{definition}[theorem]{Definition}
\theoremstyle{definition}
\newtheorem{remark}[theorem]{Remark}

\newtheorem{example}[theorem]{Example}

\DeclareMathOperator{\diam}{diam}

\renewcommand{\setminus}{\smallsetminus}

\newcommand{\N}{\mathbf{N}} %:P
\newcommand{\Z}{\mathbf{Z}}

\newcommand{\Sym}{\operatorname{Sym}}

\newcommand{\Supp}{\operatorname{Supp}}
\def\linf{\ell^\infty}
\def\Ell{\mathrm L}

\begin{document}
\title{Coarse amenability at infinity}

%%% Author 

\author{Thibault Pillon}
\thanks{The author has been supported by the European Research Council consolidator grant 614195.}
\address{Analysis Section, KU Leuven, Celestijnenlaan 200b, Box 2400\\ 3001 Leuven, Beligum}
\email{thibault.pillon@kuleuven.be}

%%%% Abstract %%%%
\begin{abstract} 
We define two different weakenings of coarse amenability (also known as Yu's property A), namely fibred coarse amenability and coarse amenability at infinity. These two properties allow us to prove that a residually finite group is coarsely amenable if and only if some (or all) of its box spaces satisfy the weak properties. We then elaborate on a result of Willett by showing that graphs with large girth always satisfy fibred coarse amenability. Finally, we discuss some examples and counter-examples to these properties.

\end{abstract}

%\date{\today}
\renewcommand{\keywordsname}{Key words}
 %Keywords
%\subjclass[2010]{AAXXBB, CCYYDD} %2010 MSC

\maketitle
\setcounter{tocdepth}{1} %Depth of the table of content
{ %This ensures the table of content to stay black even if its entries are links
%\hypersetup{linkcolor=black}
%\tableofcontents
}

%%%%%%%%%%% INTRODUCTION %%%%%%%%%%%%
\section{introduction}
The idea to associate to a group, or to a group action, a geometric object capturing some properties of the group, or the action is very well-spread. Probably the most famous of such objects is the Cayley graph associated to any group together with a generating set. In this paper, we are mainly interested in box spaces.

Given a finitely generated, residually finite group $G$, a \emph{filtration} of $G$ is a family of normal finite index subgroups $\{N_i\}_{i\ge 0}$ such that $\cap_i N_i=\{e\}$.
We will denote by $G_i$ the quotient group $G/N_i$ and if $S$ is a generating set of $G$, we define $X_i$ to be the Cayley graph of $G_i$ with respect to the image of $S$.

\begin{definition}\label{Def:BoxSpace}
The \emph{box space} $\square_{\{N_i\}} G$  associated to $G$ and ${N_i}$ is the metric space $(\sqcup_i X_i,d)$ where $d$ is defined by
\[d(x,y)=
\begin{cases}
d_{X_i}(x,y),& \text{if }x,y\in X_i\\
\max(\diam(X_i),\diam(X_j))+1, &\text{if } x\in X_i, y\in X_j, i\neq j
\end{cases}\]
\end{definition}

Observe that we don't put $S$ as a subscript of $\square_{\{N_i\}}G$. This is because, the identity map on $\square_{\{N_i\}}G$ is naturally abilipschitz equivalence between the metrics induced by two different generating sets. Since most properties we consider are bilipschitz--invariant (and more generally, quasi--isometry-invariant or coarse invariants), this is not an issue. This is not at all true of two box spaces arising from different filtrations (see e.g. \cite{KV17,DK17}). However, we will always drop the $\{N_i\}$ subscript when there is no confusion about the filtration.

The first apparition of a box space can be traced back to Margulis (\cite{Mar73}). He proved that if $G$ has property (T), the family of graphs $X_i$ is a family of expanders. In recent years, a lot of results were obtained linking properties of a residually finite group $G$ and properties of its box-spaces. Margulis original implication turns into an equivalence when considering property ($\tau$) as a weakening of property (T). Unpublished work of Guentner established that $G$ is amenable if and only if its box-spaces are coarsely amenable. He also proved that if a box space of $G$ is coarsely embeddable then $G$ enjoys the Haagerup property (See \cite{Roe} for the proofs of these statements). Later \cite{CWW} established that the Haagerup property for $G$ is actually equivalent to all (or equivalently some) box-spaces of $G$ admitting a fibred coarse embedding into a Hilbert space, a weak notion of coarse embeddability introduced in \cite{CWY}. This was generalised by Arnt (\cite{Arn16}) for groups admitting a fibred coarse embedding into an $\Ell^p$-space. Compounding all these results, we get a dictionary translating group properties into coarse geometric versions of the same properties. The objective of this note is to add one more line to this dictionary. We answer the following question:

\begin{question}\label{BasicQuestion}
Is there a coarse property of a metric space such that a residually finite group is coarsely amenable if and only if one or all of its box spaces enjoy this property?
\end{question}

Coarse amenability, or Yu's property A, was introduced in \cite{Yu00} as a coarse analogue of amenability. Instead of the original definition, we give a reformulation due to Higson and Roe (\cite{HR}).
Recall that a metric space $X$ is \emph{uniformly discrete} if there exists $\varepsilon>0$ such that $\forall x,y\in X, d(x,y)<\varepsilon\Rightarrow x=y$. It has \emph{bounded geometry} if for every $R\ge0$ there is a uniform bound on the number of elements in any ball of radius $R$. We denote by $B_x(R)$ the closed ball of radius $R$ centred at $x$.

\begin{definition}[\cite{HR}]\label{Def:CoarseAmenability}
Let $X$ be a uniformly discrete metric with bounded geometry. $X$ is \emph{coarsely amenable} if for all $R,\varepsilon>0$, there exists $S>0$ and a family $\{f_x\}_{x\in X}$ of finitely supported probability measures on $X$ such that

\begin{enumerate}
\item $\Supp(f_x)\subset B_x(S),$
\item $d(x,y)\le R\Rightarrow \|f_x-f_y\|_1\le \varepsilon.$
\end{enumerate} 
\end{definition}

In section 2. we introduce the notion of \emph{fibred coarse amenability}, this definition is inspired by fibred coarse embeddability and, very like its model, is too long to state in this introduction. We prove that this property gives a positive answer to Question \ref{BasicQuestion}. In section 3. we extract from fibred coarse amenability the following definition:
\begin{definition}\label{Def:FibredCoarseAmenability}
Let $X$ be a uniformly discrete metric space with bounded geometry. $X$ is \emph{coarsely amenable at infinity} if for all $R,\varepsilon>0$, there exists $S\ge 0$ such that for all $L\ge0$, there exist a finite subset $K_L\subset X$ with the property that all finite subsets $C\in X\setminus K_L$ of diameter at most $L$ satisfy Definition \ref{Def:CoarseAmenability} with parameters $R,\varepsilon,S)$.
\end{definition}

Summarising our results from Sections 2. and 3. we get the following theorem:

\begin{theorem}\label{Thm:Main}
Let $G$ be a finitely generated residually finite group. The following are equivalent:
\begin{enumerate}
\item $G$ is coarsely amenable.
\item All (or equivalently one) box spaces of $G$ are fibred coarsely amenable.
\item All (or equivalently one) box spaces of $G$ are coarsely amenable at infinity.
\end{enumerate}
\end{theorem}

In section 4., we prove that family of graphs with large girth are fibred coarsely amenable. This is contrast with a result of Willett 
 (\cite{Wi11}) showing that these spaces are never coarsely amenable. We also provide easy constructive examples of spaces which are not coarsely amenable at infinity. On the early stages of this project, it was not known to us whether non-coarsely amenable residually finite groups actually exist. Fortunately Osajda (\cite{Osa18}) constructed such groups.\\
 
{\bf Acknowledgements.} Part of this work was done during stays at the University of Geneva and the University of Vienna and the author wishes to thank Tatiana Nagnibeda and Goulnara Arzhantseva for inviting him. We would also like to thank Ana Khukhro who suggested the original question, Piotr Nowak, who pointed out Proposition \ref{Prop:DouglasNowak}, Damian Osajda for essentially proving that our theorems are not void statements and Martin Finn-Sell for the non-measurable amount of interest and help he provided to this project.
 
 \section{Fibred coarse amenability}
 
In the sequel $G$ will always be a finitely generated residually finite group equipped with the word metric associated to some finite generating set $S$, we will denote by $|g|$ the word-length of $G$. Let $\{N_i\}_{i\ge0}$ be a filtration of $G$ and as before, denote by $G_i$ the quotient group $G/N_i$, by $p_i$ the quotient map $G\twoheadrightarrow G_i$ and by $X_i$ the Cayley graph of $G_i$ with respect to the generating set $p_i(S)$.
From the definition of a filtration, we get the following important properties.

 \begin{itemize}

\item $|X_i|\to \infty$ when $i\to\infty$.
\item $\forall L\ge 0$, there exists $I\ge 0$ such that $p_i$ is \emph{$L$-isometric} $\forall i\ge I$. Meaning that $p_i$ is an isometry onto its image when restricted to balls of radius $L$.

 \end{itemize}

Our goal is to project coarse amenability of $G$ into a nice property  of $\square_{\{N_i\}}G$. To do so, we take inspiration form Guentner's idea to project amenability of $G$ to coarse amenability of the box space. We quickly sketch a proof of the easy direction of this theorem.

Fix, $R,\varepsilon>0$ and using F\o lner's criterion, let $F$ be a finite subset of $G$ such that
\[
|F\triangle gF|\le \varepsilon |F|,\, \forall g\in G \text{ with }|g|\le R.
\]
For each $x\in X_i\subset \square G$ fix $g\in G$ such that $p_i(g)=x$ and define $f_x\in P(X_i)$ by 
\[ f_x(y)=\frac{|p_i^{-1}(\{y\})\cap F|}{|F|.}\]
Check that for $i$ sufficiently large, the maps $f_x$ satisfy the conditions in Definition \ref{Def:CoarseAmenability} and that $f_x$ does not depend on the choice of $g$.

 The first obstacle to reproducing this argument in our case is that coarse amenability is not an equivariant property. Given $x\in\square G$, the maps $f_g$ and $f_h$ witnessing coarse amenability of two different representatives $f$ and $g$ of $x$ might considerably differ. Hence, we need to formulate coarse amenability in an equivariant way. Fortunately this has been done in \cite{DN}.

Consider the vector space $C(G,\ell^\infty(G))$ of functions $\xi:G\to\linf(G), x\mapsto \xi_x$ with finite support. Equip this space with the norm
\[
\|\xi\|_u=\left\|\sum_{x\in G} |\xi_x|\right\|_\infty=\sup_{x\in G}\sum_{g\in G} |\xi_g(x)|
\]
The completion of $C(G,\ell^\infty(G))$ with respect to this norm is denoted by $\ell_u(G)$ and known as the \emph{uniform convolution algebra} of $G$. $G$ naturally acts diagonally on $\ell_u(G)$ by $(g\cdot \xi)_x(y)=\xi_{xg}(yg)$. We say that $\xi\in\ell_u(G)$ is an \emph{$\linf(G)$-valued probability measure} if $\xi_x(y)\ge0\ \forall x,y\in G$  and \sloppy${\sum_{x\in G}\xi_x\equiv 1}$.

\begin{proposition}\label{Prop:DouglasNowak}\cite{DN}
A finitely generated group $G$ is coarsely amenable if for every $R,\varepsilon\ge 0$, there exists $\xi\in \ell_u(G)$ such that
\begin{enumerate}
\item $\xi$ is finitely supported,
\item $\xi$ is an $\ell^\infty(G)$-valued probability measure,
\item $\|\xi-g\cdot\xi\|_u\le\varepsilon$ whenever $|g|<R$.
\end{enumerate}
\end{proposition}

With this formulation of coarse amenability in hand, one could think of projecting translate of $\xi$ onto the box space to obtain a family of $\linf(G)$-valued probability measures on $\square G$. This approach still fails at least for two different reasons. Firstly, a careful observation shows that replacing probability measures on $X$ by $\linf(G)$-valued probability measure would not lead to a weaker definition\footnote{ In fact, such properties have already been considered as equivalent reformulations of coarse amenability. See \cite{Tu} for example, where it is shown that coarse amenability for a space $X$ can be phrased in terms of functions in $\ell^1(X)$, $\ell^1(X\times\N)$, $\ell^2(X)$ or $\ell^2(X\times\N)$.}. Secondly, given $x\in X_i\subset\square G$ and $g\in G$ a representative of $x$, the projection of $g\xi$ onto $X_i$ would still depend on the choice of the representative $g$. Indeed, although two different projections would now have the same support, their value in $\linf(G)$ would differ by a permutation of the index set $G$. However, the quotient maps $p_i$ being eventually $L$-isometric for any $L\ge 0$, it is always possible to make a coherent choice of representatives on finite subsets of a given diameter, up to forgetting the first few quotients $X_i$. This suggests the next definition, which is greatly inspired by the definition of a fibred coarse embedding (\cite{CWW,CWY}). From now on, we equip any function space of the form $C(X,\linf(\N))$ with the uniform norm
\[\|\xi\|_u=\left\|\sum_{x\in X}|\xi_x|\right\|_\infty\].
 
\begin{definition}\label{Def:FibredCoarseAmenability}
Let $X$ be a uniformly discrete metric space of bounded geometry. We say that $X$ is \emph{fibredly coarsely amenable} if there exists a field of Banach spaces $\left(\linf_x\right)_{x\in X}$ on $X$, where each $\ell^\infty_x$ is isomorphic to $\linf(\N)$, and $\forall R,\varepsilon\ge0$, there exists $S\ge0$ and a family of functions $\left(\xi^x\right)_{x\in X}$, $\xi^x:\:X\rightarrow\linf_x$ satisfying
 \begin{enumerate}
\item $\Supp(\xi^x)\subset B(x,S)$,
\item $\forall x\in X,$ $\xi^x$ is an $\linf_x$-valued probability measure.
\end{enumerate}
And such that for every $L\ge0$, there exists a finite subset $K_L\subset X$ such that for any finite subset $C\subset X\setminus K_L$ of diameter at most $L$, there exists a trivialisation  
\[
t_C: \bigsqcup_{x\in C}\linf_x\rightarrow C\times\linf(\N)
\]
satisfying the following: 
\begin{enumerate}
\setcounter{enumi}{2}
\item $\forall x\in C,\ t_C|_{\linf_x}$ is a linear invertible isometry $\linf_x\rightarrow\linf(\N)$ arising from some permutation\footnote{ Due to the Banach-Lamperti classification of isometries of $\ell^\infty$-spaces, this is a very weak requirement. It mainly insures that $t_C(x)\xi^x$ is an $\ell^\infty$-valued probability measure. Otherwise, we should use $|t_C(x)\xi^x|$ in the next item.} of the base set $A$. We set $t_C(x):=t_C|_{\linf_x}$.
\item For all $x,y\in C$ with $d(x,y)\le R$, $\|t_C(x)\xi^x-t_C(y)\xi^y\|_u\le\varepsilon$. Where $t_C(x)\xi^x:\: X\to\linf$ is defined by $(t_C(x)\xi^x)_z=t_C(x)(\xi^x_z)$.
\item For all $C_1,C_2\in X\setminus K_L$ of diameter at most $L$, $t_{C_1}(x)\circ t_{C_2}^{-1}(x)$ does not depend on $x\in C_1\cap C_2$. We denote this isometry by $t_{C_1C_2}$.
\end{enumerate}
\end{definition}
 
As desired, we get the following theorem :
\begin{theorem}\label{Thm:FibredCoarseAmenability} Let $G$ be a finitely generated residually finite group. The following statements are equivalent:
\begin{enumerate}
\item $G$ is coarsely amenable.
\item There exists a filtration $\{N_i\}$ of $G$ such that $\square G$ is fibred coarsely amenable.
\item For all filtrations of $G$, $\square G$ is fibred coarsely amenable.
\end{enumerate}
\end{theorem}

\begin{proof}~\\
\noindent{\it''1.$\Rightarrow$ 3.'':} We already saw that Definition \ref{Def:FibredCoarseAmenability} was designed to render this implication true. We make the details precise. Fix a filtration $\{N_i\}$ of $G$, fix $R,\varepsilon\ge 0$ and let $\xi\in\ell_u(G)$  be given, with respect to $R$ and $\varepsilon$ by Proposition \ref{Prop:DouglasNowak}.

For $x\in X_i\subset\square G$, fix an arbitrary $g_x\in G$ such that $p_i(g_x)=x$ and define $\xi^x$ by
\[
 \xi^x_y=
\begin{cases}\sum_{h:p_i(h)=y}\limits(g_x\cdot\xi)_h&\text{if }y\in X_i,\\
0&\text{otherwise}.\end{cases}
\]

Observe that the summation is well-defined since $\xi$ is finitely supported, that $\xi^x$ is indeed a finitely supported $\linf_x$-valued probability measure and that $\Supp(\xi^x)\subset B_x(S)$ for $S$ such that $\Supp(\xi)\subset B_e(S)$, where $e$ denotes the neutral element of $G$. Moreover, as soon as $i$ is large enough, there is a one-to-one correspondence between the support of $g_x\cdot\xi$ and the support of $\xi^x$ in $X_i$. Finally, observe that if another element $g_x'$ is chosen as a representative of $x$, the function $\xi^x_y$ varies only by an isometry of $\linf(G)$ (not depending on $y$). Precisely, if we denote by $\eta^x$ the function obtained using $g'_x$ as a representative, we get
$$\eta^x_y=\rho_{g_x^{-1}g_x'}\xi^x_y,\quad\forall y\in{X_i}$$
where $\rho_g$ denotes the right regular representation of $G$ onto $\linf(G)$. 

Now, for $L\ge0$, set $K_L=\sqcup X_j$, where the union is taken over all (but finitely many) $j$'s such that $p_j$ is not $L+S$-isometric. For a given subset subset $C$ of $X\setminus K_L$ of diameter less than $L$, observe that the definition of the metric $d$ on $\square G$ in Definition \ref{Def:BoxSpace} implies that $C\subset X_i$ for some $i\ge0$. Since $p_i$ is $L$-isometric, choose $C'\subset G $ such that $p_i|_{C'}$ is an isometry onto $C$ and denote by $g_{x,C}$ the unique representative of $x$ contained in $C'$. Define $t_C$, as in Definition \ref{Def:FibredCoarseAmenability} by
\[t_C(x)=\rho_{g_{x,C}g_x^{-1}}.\]
From the construction so far, we straightforwardly obtain
\[
\|t_C(x)\xi^x-t_C(y)\xi^y\|_u=\|g_{x,C}\xi-g_{y,C}\xi\|_u\le \varepsilon
\]
when $|g_{x,C}^{-1}g_{y,C}|=d(x,y)\le R$.

Finally, given $C_1$ and $C_2$ such that $C_1\cap C_2\neq \emptyset$, observe that there exists a unique $g_{C_1C_2}\in G$ such that $g_{x,C_1}=g_{C_1,C_2}g_{x,C_2}$ for all $x\in C_1\cap C_2$. We get that
\[
t_{C_1}(x)\circ t_{C_2}^{-1}(x)=
\rho_{g_{x,C_1}g_x^{-1}}
\left(\rho_{g_{x,C_2}g_x^{-1}}\right)^{-1}
=\rho_{g_{x,C_1}g_{x,C_2}^{-1}}
=\rho_{g_{C_1,C_2}}
\]
 which indeed, does not depend on $x\in C_1\cap C_2$.
 
\noindent{\emph{''3$\Rightarrow$2''}} is immediate. \\

\noindent{\emph{''2$\Rightarrow$1''}:}
Fix $R,\varepsilon>0$ and let $\left(\xi^x\right)_{x\in X}$ and $S\ge0$ be given by fibred coarse amenability. Choose $i$ large enough so that $p_i$ is $S$-isometric and such that ${X_i\cap K_S=\emptyset}$ and define $L=L(i)$ to be the largest integer for which $p_i$ is $L$-isometric. Later on, we will let $i$ tend to infinity and both conditions will be eventually always satisfied. To keep our notation simple, for $x,y\in X_i$, denote  by $t _x$ (respectively $t_{x,y}$), the trivialization $t_{B_x(L)}$ (respectively $t_{B_x(L)B_y(L)}$).

Since $p_i$ is $L$-isometric with $L\ge S$ , $\Supp (\xi^x)\subset X_i$ for all $x\in X_i$. This allows to see $\xi^x$ as an element of $C(X_i,\linf_x)$. We wish to combine all the information contained in the vectors $\xi^x$, for $x\in X_i$ into a single vector in a bigger space. We consider the space $C(X_i,\linf(X_i\times\N)$ equipped with the natural uniform norm $\|\cdot\|_u$ and define a partial action of $G_i$ onto it by
\[
(g\cdot\zeta)(y)(x,-)=t_{x,xg}\zeta(yg)(xg,-).
\]
 This expression is defined when $B_x(L)\cap B_{xg}(L)\neq\emptyset$, i.e. when $|g|\le 2L$, and it satisfies 
$$g\cdot (h\cdot \zeta)=(gh)\cdot \zeta$$
whenever $|g|,|h|,|gh|\le 2L$ (this is due to the relation \sloppy ${t_{x,xg}\circ t_{xg,xgh}=t_{x,xgh}}$).

Now, define $\eta\in C\left(X_i,\linf(X_i\times\N)\right)$ by
\[
\eta(y)(x,-)=t_x(x)\xi^x_{xy^{-1}}
\]
The vector $\eta$ is $\varepsilon$-invariant under the partial action of the ball of radius $R$ in $G_i$ in the following sense. Let $g\in G_i$, $|g|\le R$,
\begin{eqnarray*}
\|\eta-g\cdot \eta\|_u&=& \left\|\sum_{y\in X_i} \left|\eta(y)-(g\cdot\eta)(y)\right|\right\|_{\linf(X_i\times\N)}\\
&=&\sup_{x\in X_i}\left\|\sum_{y\in X_i} \left|\eta(y)(x,-)-(g\cdot\eta)(y)(x,-)\right|\right\|_\infty\\
&=&\sup_{x\in X_i} \left\|\sum_{y\in X_i}\left|\eta(y)(x,-)-t_{x,xg}\eta(yg)(xg,-)\right|\right\|_\infty\\
&=&\sup_{x\in X_i} \left\|\sum_{y\in X_i}\left|t_x(x)\xi^x_{xy^{-1}}-t_{x,xg}t_{xg}(xg)\xi^{xg}_{xgg^{-1}y^{-1}}\right|\right\|_\infty\\
&=&\sup_{x\in X_i} \left\|\sum_{y\in X_i}\left|t_x(x)\xi^x_{xy^{-1}}-t_x(xg)\xi^{xg}_{xy^{-1}}\right|\right\|_\infty\\
&=&\sup_{x\in X_i}\left\|t_x(x)\xi^x-t_x(xg)\xi^{xg}\right\|_u\le\varepsilon
\end{eqnarray*}

According to Definition \ref{Def:FibredCoarseAmenability}, when $g\in G_i$, $|g|\le 2L$, the isometry $t_{x,xg}$ arises from a permutation $\sigma_g\in\Sym(\N)$ in the sense that $(t_{x,xg}\xi)(n)=\xi(\sigma_g^{-1}(n))$. The action of $G_i$ on $C\left(X_i,\linf(X_i\times \N)\right)$ can then be rewritten as
\[
(g\cdot\zeta)(y)(x,n)=\zeta(yg)(g^{-1}\cdot(x,n)),
\]
Where the formula $g\cdot(x,n)=(xg^{-1},\sigma_g(n))$ defines a partial action of $G_i$ on $X_i\times\N$. We now lift the values of $\eta(y)$ to a vector in $\linf(G)$ by using the map ${x\mapsto p_i(x)^{-1}\cdot z_i}$, where $z_i$ is any fixed base point of  $X_i\times\N$. 

For $x,y\in G$, consider the sequence $\left(a_{x,y,i}=\eta(p_i(y))(p_i(x)^{-1}z_i)\right)_{i\in\N}$. This sequence is well-defined for all $i$ large enough so that $|x|\le 2L(i)$ and takes values in the compact interval [0,1]. Using Cantor's diagonal argument and passing to a subsequence, we can assume that the sequence $(a_{x,y,i})_{i\in\N}$ converges for all $x,y\in G$. Define $\xi\in \ell_u(G)$ by 
\[\xi_y(x)= \lim_{i\to\infty} a_{x,y,i} .\]
For $|g|\le R$, we compute:
\begin{eqnarray*}
\|\xi-g\cdot\xi\|_u&=&\sup_{x\in G}\sum_{y\in G}\left|\xi_y(x)-\xi_{yg}(xg)\right|\\
&=&\sup_{x\in G}\sum_{y\in G}\lim_{i\to\infty}\left|\eta(p_i(y))(p_i(x)^{-1}z_i)-\eta(p_i(yg))(p_i(xg)^{-1}z_i)\right|\\
&=&\sup_{x\in G}\lim_{i\to\infty}\sum_{y\in G_i}\left|(\eta(y)(p_i(x)^{-1}z_i)-\eta(yp_i(g))(p_i(g)^{-1}p_i(x)^{-1}z_i)\right|\\
&=&\sup_{x\in G}\lim_{i\to\infty}\sum_{y\in G_i}\left|\eta(y)(p_i(x)^{-1}z_i)-(p_i(g)\cdot\eta)(y)(p_i(x)^{-1}z_i)\right|\\
&\le& \lim_{i\to\infty}\left\|\sum_{y\in G_i}\eta(y)-(p_i(g)\cdot\eta)(y)\right\|_\infty\\
&\le &\varepsilon
\end{eqnarray*}
Very similar computation shows that $\xi$ is indeed an $\linf(G)$-valued probability measure, which concludes the proof.

\end{proof}

\section{Coarse amenability at infinity}
 
There are two main ingredients in Definition \ref{Def:FibredCoarseAmenability}. First, the maps $\xi^x$ together with the trivialisation $t_C$ are witnesses of coarse amenability for  the subset $C$. Secondly, the existence of the maps $t_{C_1C_2}$ allows to "patch-up" these witnesses in a coherent manner. We'll see that the second condition is not actually needed to serve our purpose.
We will rely on the following definition and result, due to Brodzki, Niblo and Wright (\cite{BNW}). Let $X$ be a uniformly discrete  metric space of bounded geometry.

\begin{definition}
$X$ is \emph{locally coarsely amenable} if for all $R,\varepsilon$, there exists $S\ge0$ such that every finite subset of $X$ satisfies Definition \ref{Def:CoarseAmenability} with respect to the parameters $R,\varepsilon$ ans $S$.
\end{definition}

\begin{theorem}[\cite{BNW}]\label{Thm:LocalCoarseAmenability} $X$ is coarsely amenable if and only if $X$ is locally coarsely amenable.
\end{theorem}

The following definition, which was already mentioned in the introduction is a natural weakening of uniform coarse amenability.

\begin{definition}
$X$ is \emph{coarsely amenable at infinity} if for every $R,\varepsilon\ge0,$ there exists $S\ge0$ such that for all $L\ge0$, there exists a finite subset $K_L$ of $X$ such that every subset $C\subset X\setminus K_L$ of diameter at most $L$ satisfies the requirement of Definition \ref{Def:CoarseAmenability} for the parameters $R,\varepsilon,S$.
\end{definition}

Local coarse amenability requires that all finite subspaces enjoy coarse amenability uniformly. Coarse amenability at infinity requires that almost all finite subspaces enjoy coarse amenability uniformly, up to forgetting some finite subset of the space depending on the scale of the finite subspace in consideration. In other words, $X$ is coarsely amenable at infinity if, at each scale, only finitely many finite subspaces behave badly.

The following proposition shows that uniform coarse amenability is a weakening of fibred coarse amenability.

\begin{proposition}\label{Prop:FCAimpliesCAI}
If $X$ is fibredly coarsely amenable then $X$ is coarsely amenable at infinity.
\end{proposition}
\begin{proof}
 Fix $R,\varepsilon\ge0$ and $L\ge0$ and let $(\xi^x)_x\in X,S\ge0$ and $K_L\subset X$ be given by fibred coarse amenability. Fix $C\subset X\setminus K_{L+S}$ of diameter at most $L$ and let $t_C$ be the trivialisation associated to $C$. Now, for $x\in C$ define a probability measure $f_x$ on $X$ by $f_x(z)=t_C(x)\xi^x_z(0)$. For $x,y$ such that $d(x,y)\le R$, we have
\begin{eqnarray*}
\|f_x-f_y\|_1&=&\sum_{z\in X}|f_x(z)-f_y(z)|\\
&=& \sum_{z\in X}|t_C(x)\xi^x_z(0)-t_C(y)\xi^y_z(0)|\\
&\le& \left\|\sum_{z\in X}|t_C(x)\xi^x_z-t_C(y)\xi^y_z|\right\|_\infty\\
&=&\|t_C(x)\xi^x-t_C(y)\xi^y\|_u\le \varepsilon.
\end{eqnarray*} 
The probability measures $f_x$ act as witnesses of coarse amenability for $C$, seen as a subset of $\cup_{x\in C} \Supp{f_x}$. From the existing literature on coarse amenability (see e.g. the proof of Proposition 4.2 in \cite{Tu}) we deduce that $C$ is coarsely amenable with respect to a family of parameters $R',\varepsilon',S'$ independent of $C$.

\end{proof}

In the case of a box-space, we get a similar result as in the previous section:

\begin{theorem}\label{Thm:CoarseAmenabilityAtInfinity} Let $G$ be a finitely generated residually finite group. The following statements are equivalent:
\begin{enumerate}
\item $G$ is coarsely amenable.
\item There exists a filtration $\{N_i\}$ of $G$ such that $\square G$ is coarsely amenable at infinity.
\item For all filtrations of $G$, $\square G$ is coarsely amenable at infinity.
\end{enumerate}
\end{theorem}

\begin{proof} The proof is much easier than the proof of Theorem \ref{Thm:FibredCoarseAmenability}.

\noindent{\it''1.$\Rightarrow$ 3.'':} Fix $R,\varepsilon, L\ge0$ and define $K_L$ as the union of those $X_i$ for which $p_i$ is not $L$-isometric. Outside $K_L$ all subsets of diameter at most $L$ are isometric to a subset of the ball of radius $L$ in $G$. Since $G$ is coarsely amenable and hence locally coarsely amenable, there exists $S\ge0$ such that any finite subset of $G$ is coarsely amenable with respect to the parameters $R,\varepsilon,S$ .\\

\noindent{\emph{''3$\Rightarrow$2''}} is immediate. \\

\noindent{\emph{''2$\Rightarrow$1''}:} We establish that $G$ is uniformly coarsely amenable. Fix $R,\varepsilon>0$ and a finite subset $C\subset G$ and let $S>0$ be given by coarse amenability at infinity of $\square G$. Choose $L$ such that $C$ is a subset of a ball of radius $L$ in $G$ and choose $i$ large enough so that $p_i$ is $L$-isometric and that $X_i\cap K_L=\emptyset$. $C$ is isometric to $p_i(C)$ and by coarse amenability at infinity of $\square G$, $C$ is coarsely amenable with respect to the parameters $R,\varepsilon,S$. 

\end{proof}

In view of Proposition \ref{Prop:FCAimpliesCAI}, it is natural to wonder whether fibred coarse amenability and coarse amenability at infinity are actually the same property. In full generality, this question seems very difficult to attack, but some cases are easily tackled.

\begin{proposition}~
\begin{enumerate}
\item For a finitely generated group, coarse amenability, fibred coarse amenability and coarse amenability at infinity are equivalent.
\item The same holds for any uniformly discrete metric space $X$ of bounded geometry for which the orbit of any point of $X$ under the isometry group of $X$ is infinite.
\item For a box space $\square G$, fibred coarse amenability and coarse amenability at infinity are equivalent.
\end{enumerate}
\end{proposition}
\begin{proof}~

1. This is a special case of item 2.

2. The only non-trivial implication is that coarse amenability at infinity implies coarse amenability. Fix $R,\varepsilon>0$ and $C$ a finite subset of $X$. Let $L=\operatorname{diam}(C)$ and let $K_L$ and $S$ be given by coarse amenability at infinity. Since $K_L$ is finite, it is always possible to find an isometry $\varphi$ of $X$ such that $C\cap K_L=\emptyset$. We deduce that $C$ is coarsely amenable for the parameters $R,\varepsilon,S$. Hence $X$ is locally coarsely amenable.

3. This follows from Theorems \ref{Thm:FibredCoarseAmenability} and \ref{Thm:CoarseAmenabilityAtInfinity}.
\end{proof}

\begin{remark}
A similar situation occurs when considering the Haagerup property for a group $G$. We already mentioned that a residually finite group has the Haagerup property if and only if its box-spaces admit fibred coarse embeddings into Hilbert spaces. In \cite{WY14}, the authors introduced a notion of boundary a-(T)-menability and obtained the exact same result. Similarly to our setting, any fibredly coarsely embeddable space is easily seen to be boundary a-(T)-menable, but the converse is unknown. 
\end{remark}

\section{Examples}

\subsection{Spaces of graphs with large girth}~\\

Recall that the {\it girth} $g(X)$ of a graph $X$ is the length of its smallest simple cycle and that a family of graphs $\{X_i\}_{i\in\N}$ has {\it large girth} if $g(X_i)$ tends to infinity with $n$. Given a family $\{X_i\}_{i\in\N}$ of finite metric spaces, define a metric on $\sqcup X_i$ similarly as in Definition \ref{Def:BoxSpace}. The metric space obtained this way is called the \emph{coarse disjoint union} of the $X_i$'s.

\begin{theorem}\label{Thm:FCALargeGirth}
Let $\{X_i\}_{i\in N}$ be a family of graphs with large girth such that every vertex in ${X_i}$ has degree at least 2. The coarse disjoint union $X=\sqcup_{i\in\N} X_i$ is fibredly coarsely amenable.
\end{theorem}

This result is in contrast with the result of Willett (\cite{Wi11}) that such spaces are never coarsely amenable. 

\begin{proof}

First, we consider coarse amenability at the level of the universal covers of the graphs $X_i$. Denote by $\tilde{X_i}$ the universal cover of $X_i$ and by $p_i:\tilde{X_i}\to X_i$ the associated covering map. $\tilde X_i$ is an infinite tree with no leaves (vertices of degree 1) because every $x\in X_i$ has degree at least 2. Also, any ball of radius $r<g(X)/2$ in $\tilde X_i$ is isometrically mapped to the ball of radius $r$ around $p_i(x)$. Since the girths of the graphs $X_i$ are increasing, the maps $p_i$ are eventually $L$-isometric for any $L\ge 0$. Trees are coarsley amenable metric spaces and, furthermore,  from the proof of this fact (see e.g. Example 4.1.5 in \cite{NY}) we actually get that the family $\tilde{X_i}$ is uniformly coarsely amenable in the sense that given $R,\varepsilon\ge0$, $S$ in Definition \ref{Def:CoarseAmenability} can be chosen independently of $i$. Fix $R,\varepsilon\ge0$ and let $(f_x)_{x\in\tilde{X_i}}$ be the probability measures witnessing coarse amenability for $\tilde X_i$.

In a manner similar to Proposition \ref{Prop:DouglasNowak}, we would like to collect the data contained in the maps $f_n$ in a single map with value in a single $\ell^\infty$-space. To do this, consider the fundamental group $G_i$ of the graph $X_i$ together with its natural action on $\tilde{X_i}$. Fixing $x\in \tilde{X_i}$ define a map
$\eta^x:\tilde{X_i}\to\ell^\infty(G_i)$ by $\eta^x_y(g)=f_{g\cdot x}(g\cdot y)$. The map $\eta^x$ satisfies the following:
\begin{enumerate}
\item $\eta^x$ is an $\ell^\infty(G_i)$-valued probability measure,
\item $\Supp\eta^x\subset B_x(S)$,
\item $d(x,y)\le R \Rightarrow||\eta^x-\eta^y||_u\le\varepsilon$
\end{enumerate}

We are now in shape to prove fibred coarse amenability for the space $X$ similarly to the first implication of Theorem \ref{Thm:FibredCoarseAmenability}. For every $i$, fix an enumeration of $G_i$ and implicitly identify $\linf(G_i)$ with $\linf(\N)$. We do not define the vectors $\xi^x$ as in Definition \ref{Def:FibredCoarseAmenability} but rather we immediately  define the vector $t_C(x)\xi^x$ instead. This is of no consequence as soon as we can give coherent definition of the maps $t_{C_1C_2}$.

Fix $L\gg 2S+R$ and define $K_L$ as the union of all those $X_i$ such that $p_i$ is not $L$-isometric. Now, for $C\subset X_i\subset X\setminus K_L$ of diameter at most $L$, find $\tilde C\subset \tilde{X_i}$ such that $p_i$ maps $\tilde C$ isometrically to $C$. Define the vectors $t_C(x)\xi^x$ by projecting the vectors $\eta^{\tilde x}$ from $B_{\tilde x}(S)$ to $B_x(S)$. The vectors $\xi^x$ naturally inherits properties (1) to (3) of the vectors $\eta^{\tilde x}$\footnote{Note that the condition $L\gg 2S+R$ insures that for $x,y\in X_i$, $d(x,y)\ge R$, $p_i$ is an isometry when restricted on some ball containing both the supports of $\eta^x$ and $\eta^y$.}

Now let $C_1,C_2$ be two subsets of $X\setminus K_L$ of diameter at most $L$ and for $x\in C_1\cap C_2$ denote by $\tilde x_1$ (resp. $\tilde x_2$) the unique element of $\tilde C_1$ (resp. $\tilde C_2$) mapped by $p_i$ to $x$. By the theory of covering maps, there exists a unique $g\in G_i$ such that $g\cdot \tilde x_1=\tilde x_2$ for every $x\in C_1\cap C_2$. We see that $t_{C_1}(x)\xi^x=\lambda_g\cdot t_{C_2}(x)\xi^x$, with $\lambda$ the natural left regular representation of $G_i$ onto $\linf(G_i)$. We get $t_{C_1C_2}=\lambda(g)$ which concludes the proof.
\end{proof}

\subsection{Spaces non-coarsely amenable at infinity}

As the title suggests, the aim of this section is to provide the first example of spaces lacking coarse amenability at infinity. We apply a construction used in \cite{W07} and \cite{FSW}.

Let $\{X_i\}$ be a sequence of finite graphs and let $X=\sqcup X_i$ be their coarse disjoint union. Set $Y_{i,j}=X_i$ for all $j$, and let $Y=\sqcup Y_{i,j}$ be the coarse disjoint union of the $Y_{i,j}$'s.

\begin{proposition}\label{Prop:non-CAI} Using the above notation, if $Y$ is coarsely amenable at infinity, then $X$ is coarsely amenable.
\end{proposition} 

\begin{proof}
The proof is straightforward. We show that $X$ is locally coarsely amenable. Fix  $R$,$\varepsilon>0$ and let $S$ and the family of subsets $K_L$ be given by coarse amenability at infinity for $Y$. Now, given $C$ a finite subset of some $X_n$, there are infinitely many copies of $C$ inside $Y$. For $L\gg\operatorname{diam}(C)$, some copy of $C$ lies in $Y\setminus K_L$ and hence is coarsely amenable for the parameters $R,/varepsilon$ and $S$. Therefore, $X$ is uniformly coarsely amenable.
\end{proof}

This construction allows to compare coarse amenability at infinity with other coarse properties of the space.

\begin{corollary}
There exists uniformly discrete metric spaces of bounded geometry which are
\begin{enumerate}
\item coarsely amenable at infinity but non-embeddable into a Hilbert space
\item non-coarsely amenable at infinity and non-embeddable into a Hilbert space,
\item non-coarsely amenable at infinity but embeddable into a Hilbert.
\end{enumerate}
\end{corollary}

\begin{proof}~

\begin{enumerate}
\item Let $G$ be a residually finite coarsely amenable group with property (T), such as $\operatorname{SL}_n(\Z)$ and consider any box space of $G$.
\item Apply Proposition \ref{Prop:non-CAI} to the above sequence of quotients. 
\item Let $X=\sqcup_{i\in N} F_2/N_i$ be a box space of the free group which is embeddable in a Hilbert space but not coarsely amenable. Such box spaces were constructed in \cite{AGS12}. Apply the previous construction to produce a space $Y$. By a result of Ostrovskii, coarse embeddability is a local property (see \cite{Ost12} for a precise statement) and it follows that $Y$ is coarsely embeddable into a Hilbert space. But by the previous corollary, $Y$ in not coarsely amenable at infinity.
\end{enumerate}
\end{proof}

Less ad-hoc examples can be extracted from the litterature.

\begin{example}
\begin{enumerate}
\item Osajda (\cite{Osa18}) built the first examples of residually finite non-coarsely amenable groups $G$. Box-spaces $\sqcup G/N_i$ of these groups naturally fail coarse amenability at infinity 
\item Moreover, Mimura and Sakho (\cite{MS18} managed to find alternative generating sets $T_i$ of the groups $G/N_i$ such that the sequence $(G/N_i,T_i)$ converges to an amenable group $G_T$ in the space of marked groups.  From the definition of convergence it can be seen that balls of a fixed radius in the Cayley graph of $(G/N_i,T_i)$ are eventually isometric to a ball in $G_T$. It is straightforward to see that this implies that the coarse disjoint union $\sqcup \operatorname{Cay}(G_i,T_i)$ is coarsely amenable, in contrast the first item.
\item Alekseev and Finn-Sell ( see \cite{AF16} Example 7.8) observed that some LEF sequences associated to non-coarsely amenable groups constructed in \cite{Osa14}\footnote{Note that these groups predates \cite{Osa18} and were not residually finite.} are not coarsely amenable at infinity. 
\end{enumerate}
\end{example}

Finally, we wonder whether results from this paper can be adapted to fit different settings. In \cite{Roe05}, given a continuous action of a group on a compact space by homeomorphisms, Roe introduced the notion of warped cone, which is a metric space of bounded geometry. He proved, under suitable conditions, that the warped cone of the action is coarsely amenable if and only if the action is amenable.

\begin{question}
Is there an analogue of Theorem \ref{Thm:CoarseAmenabilityAtInfinity} in the setting of warped cones of topological actions?
\end{question}

As an analogue of coarse amenability for the action, one could consider exactness of the reduced cross product $C^\ast$-algebra.

\begin{question}
Are there general analogues of Theorem \ref{Thm:CoarseAmenabilityAtInfinity} in the setting of sofic groups, spaces of graphs associated to sofic representations and sofic boundaries introduced in \cite{AF16}? \end{question}

Finally , we wonder whether coarse amenability at infinity translates as a nice property of the coarse boundary groupoid (see \cite{Roe} for precise definitions of coarse groupoids and \cite{FSW} and \cite{FS} for the boundary case).

\begin{question}
Let $X$ be a uniformly discrete metric space of bounded geometry. Does $C^*$-exactness of the boundary groupoid of $X$ imply that $X$ is coarsely amenable at infinity?
\end{question}

%%%%%%%% Bibliography %%%%%%%% 

\bibliography{../../MaBib}

\end{document}